\documentclass[10pt,a4paper,reqno]{amsart}
\usepackage{bbm}
\usepackage{amsmath}
\usepackage{amsfonts}
\usepackage{amssymb,bm}
\usepackage{color}

\newcommand{\rot}{\operatorname{rot}}

\hoffset=- 2cm \voffset=- 0cm 
\newtheorem{Thm}{Theorem}

\newtheorem{Rem}[Thm]{Remark}
\newtheorem{Lem}[Thm]{Lemma}

\usepackage{amssymb}
\usepackage[active]{srcltx}
\usepackage{color}

\newcommand {\p}{\partial}

\newcommand{\eq}{\begin{equation}}
\newcommand{\eeq}{\end{equation}}

\def\div{\text{\rm div\,}}

\def\O{\Omega}

\def\p{\partial}

\def\A{\bold A}
\def\B{\bold B}
\def\D{\bold D}
\def\E{\bold E}

\def\f{\bold f}
\def\F{\bold F}

\def\H{\bold H}

\def\n{\bold n}

\def\R{\Bbb R}
\def\u{\bold u}

\def\bv{\bold v}

\def\w{\bold w}

\def\0{\bold 0}

\errorcontextlines=0 \numberwithin{equation}{section}
\numberwithin{Thm}{section}

\begin{document}

\large

\author{Zhibing Zhang}
\author{Chunyi Zhao}

\address{Zhibing Zhang: School of Mathematics and Physics, Anhui University of Technology, Ma'anshan 243032, People's Republic of China; }
\email{zhibingzhang29@126.com}

\address{Chunyi Zhao: School of Mathematical Sciences, Shanghai Key Laboratory of PMMP, East China Normal University, Shanghai 200241, People's Republic of China.}

\email{cyzhao@math.ecnu.edu.cn}

\thanks{}

\title[steady MHD equations]
{Solvability of an inhomogeneous boundary value problem for steady MHD equations}



\begin{abstract}
In this paper, we consider the steady MHD equations with inhomogeneous boundary conditions for the velocity and the tangential component of the magnetic field. Using a new construction of the magnetic lifting, we obtain existence of weak solutions under sharp assumption on boundary data for the magnetic field.
\end{abstract}

\maketitle

\section{Introduction}

Let $\O$ be a bounded domain in $\mathbb{R}^3$ with $C^{1,1}$ boundary $\p\O$. The boundary $\p\O$ has a finite number of connected components $\Gamma_0,\Gamma_1,\cdots,\Gamma_m$, where $\Gamma_0$ denotes the boundary of the infinite connected component of $\mathbb{R}^3\backslash\overline{\O}$. The domain $\O$ can be multiply connected. Assume that there are disjoint $C^2$ cuts $\Sigma_1,\cdots, \Sigma_N$ such that $\O\backslash\cup_{i=1}^N \Sigma_i$ is simply connected and Lipschitz. We denote by $\mathbf{n}$ the unit outer normal on $\p\O$. We consider the following steady MHD equations
\begin{equation}\label{MHD1}
\begin{cases}
-\nu\Delta \u+(\u\cdot\nabla)\u+\nabla p-\varkappa\rot\H\times\H={\bf f,}\quad \div\u=0\quad&\text{in } \O,\\
\nu_1\rot\H-\E+\varkappa\H\times\u=\nu_1\mathbf{j,}\quad \div\H=0,\quad \rot\E=\0\quad & \text{in } \O,
\end{cases}
\end{equation}
with inhomogeneous boundary conditions
\begin{equation}\label{MHD-bc}
  \u=\mathbf{g},\quad\H\times\mathbf{n}=\mathbf{q}\qquad \text{on }\p\O.
\end{equation}
Here $\u$ is the velocity vector, $\H$ is the magnetic field, $\E=\E'/\rho_0$, $p=P/\rho_0$, where $\E'$ is the electric field, $P$ is the pressure, $\rho_0=\text{const}>0$ is the fluid density. Moreover, $\nu$ is the kinematic viscosity, $\varkappa=\mu/\rho_0$, $\nu_1=1/(\rho_0\sigma)$, where $\mu,\sigma$ are the magnetic permeability and the electric conductivity. Besides, $\f$ and $\mathbf{j}$ are given functions defined on $\O$, and $\mathbf{g}$ and $\mathbf{q}$ are given functions defined on $\p\O$.

Beginning from the pioneering works \cite{LS1960,S1960}, the solvability of boundary value problems for the MHD equations has been studied in many works, see, for example, \cite{Alekseev04,Alekseev,Alekseev16,GMP1991,ST1983} and the references therein. Precisely, Solonnikov \cite{S1960} proved the solvability of
the steady MHD equations (\ref{MHD1}) with the homogeneous boundary conditions
$$\H\cdot\mathbf{n}=0,\quad \E\times\mathbf{n}=\0\quad\text{ on }\p\O.$$
Later, Gunzburger, Meir and Peterson \cite{GMP1991} considered the corresponding inhomogeneous boundary value problem with
\begin{equation}\label{bc}
\u=\mathbf{g},\quad\H\cdot\mathbf{n}=q,\quad \E\times\mathbf{n}=\mathbf{k}\quad\text{ on }\p\O,
\end{equation}
and obtained the local (i.e., in the case of small boundary data) solvability. The global (i.e., without small boundary data condition) solvability of (\ref{MHD1}) with the boundary conditions \eqref{bc} was firstly proved by Alekseev \cite{Alekseev04,Alekseev} under the assumption that the boundary data $\mathbf{g}$ is tangential to the boundary. As for the problem \eqref{MHD1}-(\ref{MHD-bc}), Gunzburger, Meir and Peterson \cite{GMP1991} claimed that its local solvability can be similarly proved as (\ref{MHD1})-\eqref{bc}. In 2016, Alekseev \cite{Alekseev16} obtained the global solvability of \eqref{MHD1}-(\ref{MHD-bc}) under the conditions that the boundary data $\mathbf{g}\in H^{1/2}(\p\O,\Bbb R^3)$, $\mathbf{q}\in L^2(\p\O,\Bbb R^3)$ are tangential to the boundary and $\mathbf{q}$ satisfies
\begin{equation}\label{cond-q}
\div_{\p\O}\;\mathbf{q}=0, \quad \int_{\p\O}\mathbf{q}\cdot \mathbf{h}~\mathrm dS=0, \quad\forall~\mathbf{h}\in \Bbb H_N(\O).
\end{equation}
Here $\div_{\p\O}$ is the operator of surface divergence, $\Bbb H_N(\O)$ represents the harmonic Neumann fields, that is
$$
\aligned
&\Bbb H_N(\O)=\{\u\in L^2(\O,\Bbb R^3): \rot\u=\0,\; \div\u=0\text{ in $\O$,}\; \u\cdot\mathbf{n}=0 \text{ on $\p\O$}\}.
\endaligned
$$
Very recently, Alekseev and Brizitskii \cite{AB20} generalized the global solvability result for $\mathbf{g}=\0$ and $\mathbf{q}\in H^s(\p\O,\Bbb R^3)$ being tangential to the boundary and satisfying \eqref{cond-q}, where $s\in [0,1/2]$ is arbitrary. In fact, the condition $\mathbf{g}=\0$ can be replaced by the general condition that $\mathbf{g}\in H^{1/2}(\p\O,\Bbb R^3)$ is tangential to the boundary.

Let $\H_0$ be a magnetic lifting for the magnetic boundary data $\mathbf{q}$, i.e., $\H_0$ is an extension for $\mathbf{q}$ satisfying
$$\div\H_0=0\text{ in $\O$,}\quad \H_0\times\mathbf{n}=\mathbf{q} \text{ on $\p\O$}.$$
To overcome the difficulty caused by the term $\varkappa \rot \H_0 \times \widetilde \H$ in the homogenized system \eqref{MHD3}, where $\widetilde \H=\H-\H_0$, Alekseev \cite{Alekseev16} adopted a magnetic lifting satisfying the $\div$-$\rot$ system
$$ \rot\H_0=\0,\quad \div\H_0=0\text{ in $\O$,}\quad \H_0\times\mathbf{n}=\mathbf{q} \text{ on $\p\O$}.$$
Hence the extra condition \eqref{cond-q} is necessary to guarantee the above system admitting a solution. In that situation, indeed, the term $\varkappa \rot \H_0 \times \widetilde \H$ disappears, but the cost is that more unnatural restrictions are imposed on $\mathbf{q}$.

In this paper, using the construction of the hydrodynamic lifting suggested in \cite{Alekseev04} and a new construction of the magnetic lifting (see Lemma \ref{lemma-2.2}), we show the global solvability of the problem \eqref{MHD1}-(\ref{MHD-bc}) \emph{without the condition \eqref{cond-q}} by applying Schauder's fixed point theorem. We use integration by parts to deal with the term $\varkappa \rot \H_0 \times \widetilde \H$, see the key inequality \eqref{ineq3}.

Throughout this paper, we use $L^2(\O)$, $H^1(\O)$, $H^{-1}(\O)$, $H^s(\p\O)$ to denote the usual Lebesgue spaces, Sobolev spaces for scalar functions, and $L^2(\O,\Bbb R^3)$, $H^1(\O,\Bbb R^3)$,  $H^{-1}(\O,\Bbb R^3)$, $H^s(\p\O,\Bbb R^3)$ to denote the corresponding spaces of vector fields. However we use the same notation to denote the norm of both scalar functions and vector fields. For instance, we write $\|\phi\|_{L^2(\O)}$ for $\phi\in L^2(\O)$ and also $\|\u\|_{L^2(\O)}$ for $\u\in L^2(\O,\Bbb R^3)$.
We also use the following notations:
$$\aligned
&H^1_0(\div0,\O)=\{\u\in H^1_0(\O,\Bbb R^3): \div\u=0\;\text{in }\O\},\\
&H^1_{t0}(\O,\Bbb R^3)=\{\u\in H^1(\O,\Bbb R^3):  \u\times\mathbf{n}=\0\;\text{on }\p\O\},\\
&H^1_{t0}(\div0,\O)=\{\u\in H^1(\O,\Bbb R^3): \div\u=0\;\text{in }\O, \u\times\mathbf{n}=\0\;\text{on }\p\O\},\\
&H^{1/2}_{T}(\p\O,\Bbb R^3)=\{\u\in H^{1/2}(\p\O,\Bbb R^3):  \u\cdot\mathbf{n}=0\;\text{on }\p\O\}.\\
\endaligned
$$
The harmonic Dirichlet fields $\Bbb H_D(\O)$ is defined by
$$
\aligned
&\Bbb H_D(\O)=\{\u\in L^2(\O,\Bbb R^3): \rot\u=\0,\;\div\u=0\text{ in $\O$,}\; \u\times\mathbf{n}=\0 \text{ on $\p\O$}\}.
\endaligned
$$

Now our main result reads as follows.
\begin{Thm}\label{existence}
Let $\nu$, $\nu_1$ and $\varkappa$ be three positive constants. Assume $\mathbf{f}\in H^{-1}(\O,\R^3)$, $\mathbf{j}\in L^2(\O,\R^3)$ and $\mathbf{g},\mathbf{q}\in H^{1/2}_{T}(\p\O,\Bbb R^3)$. Then \eqref{MHD1}-\eqref{MHD-bc}
admits a weak solution
$$(\u,p,\mathbf{E},\mathbf{H})\in H^1(\O,\R^3)\times L^2(\O)\times L^2(\O,\R^3)\times H^1(\O,\R^3).$$
\end{Thm}

\begin{Rem}
Since the magnetic field satisfies the boundary condition $\H\times\mathbf{n}=\mathbf{q}$ on $\p\O$, $\mathbf{q}$ must satisfy the compatible condition $\mathbf{q}\cdot\mathbf{n}=0$ on $\p\O$.
So the assumption we propose on the magnetic boundary data is sharp.
\end{Rem}

\section{Proof of Theorem \ref{existence}}

In order to prove the global solvability result, we introduce two trace lifting lemmas.
\begin{Lem}$($\cite[Lemma 2.2]{Alekseev04}$)$\label{lemma-2.1}
For every vector-valued function $\mathbf{g}\in H_T^{1/2}(\p\O,\Bbb R^3)$ and every number $\varepsilon>0$, there exists a vector-valued function $\u_\varepsilon\in H^1(\O,\Bbb R^3)$ such that
$$\aligned
&\div\u_\varepsilon=0\text{ in $\O$, } \u_\varepsilon=\mathbf{g} \text{ on $\p\O$, }\\
&\|\u_\varepsilon\|_{L^4(\O)}\leq \varepsilon\|\mathbf{g}\|_{H^{1/2}(\p\O)},\\
&\|\u_\varepsilon\|_{H^1(\O)}\leq C_\varepsilon\|\mathbf{g}\|_{H^{1/2}(\p\O)},\\
\endaligned
$$
where the constant $C_\varepsilon$ depends on $\varepsilon$ and $\O$.
\end{Lem}

\begin{Lem}\label{lemma-2.2}
For every vector-valued function $\mathbf{q}\in H_T^{1/2}(\p\O,\Bbb R^3)$ and every number $\varepsilon>0$, there exists a vector-valued function $\H_\varepsilon\in H^1(\O,\Bbb R^3)$ such that
$$\aligned
&\div\H_\varepsilon=0\text{ in $\O$, } \H_\varepsilon\times\mathbf{n}=\mathbf{q} \text{ on $\p\O$, }\\
&\|\H_\varepsilon\|_{L^4(\O)}\leq \varepsilon\|\mathbf{q}\|_{H^{1/2}(\p\O)},\\
&\|\H_\varepsilon\|_{H^1(\O)}\leq C_\varepsilon\|\mathbf{q}\|_{H^{1/2}(\p\O)},\\
\endaligned
$$
where the constant $C_\varepsilon$ depends on $\varepsilon$ and $\O$.
\end{Lem}
\begin{proof}
Let $\varepsilon$ be an arbitrary positive number. Since $\mathbf{n}\times\mathbf{q}\in H_T^{1/2}(\p\O,\Bbb R^3)$, then by Lemma \ref{lemma-2.1} there exists a vector-valued function $\u_\varepsilon\in H^1(\O,\Bbb R^3)$ such that
$$\aligned
&\div\u_\varepsilon=0\text{ in $\O$, } \u_\varepsilon=\mathbf{n}\times\mathbf{q}\text{ on $\p\O$, }\\
&\|\u_\varepsilon\|_{L^4(\O)}\leq \varepsilon\|\mathbf{n}\times\mathbf{q}\|_{H^{1/2}(\p\O)},\\
&\|\u_\varepsilon\|_{H^1(\O)}\leq C_\varepsilon\|\mathbf{n}\times\mathbf{q}\|_{H^{1/2}(\p\O)},\\
\endaligned
$$
where the constant $C_\varepsilon$ depends on $\varepsilon$ and $\O$. Since
$$\|\mathbf{n}\times\mathbf{q}\|_{H^{1/2}(\p\O)}\leq C(\O)\|\mathbf{q}\|_{H^{1/2}(\p\O)},$$
we obtain
$$\|\u_\varepsilon\|_{L^4(\O)}\leq C(\O)\varepsilon\|\mathbf{q}\|_{H^{1/2}(\p\O)}.$$
Set $\H_\varepsilon=\u_{\varepsilon/C(\O)}$. Then we have
\begin{equation*}
\H_\varepsilon\times\mathbf{n}=(\mathbf{n}\times\mathbf{q})\times\mathbf{n}=\mathbf{q}\text{ on $\p\O$, } \qquad
\|\H_\varepsilon\|_{L^4(\O)}\leq \varepsilon\|\mathbf{q}\|_{H^{1/2}(\p\O)}. \hfill
\qedhere
\end{equation*}
\end{proof}
We also need the following $\div$-$\rot$ inequality, which is a consequence of  \cite[p. 209, Theorem 3]{DL1990} and \cite[p. 213, Remark 2]{DL1990}, see also \cite[Lemma 2.1]{Alekseev16}.
\begin{Lem}\label{lemma-2.3}
For any $\mathbf{B}\in H^1_{t0}(\O,\Bbb R^3)\cap \Bbb H_D(\O)^\perp$, it holds that
$$\|\mathbf{B}\|_{H^1(\Omega)}\leq C(\O)(\|\div\mathbf{B}\|_{L^2(\Omega)}+\|\rot\mathbf{B}\|_{L^2(\Omega)}).$$
\end{Lem}

Now we are in a position to prove the global solvability result.

\begin{proof}[Proof of Theorem \ref{existence}]
By eliminating $\E$, \eqref{MHD1}-\eqref{MHD-bc} turns into the following form
\begin{equation}\label{MHD2}
\begin{cases}
-\nu\Delta \u+(\u\cdot\nabla)\u+\nabla p-\varkappa\rot\H\times\H={\bf f}\quad&\text{in } \O,\\
\rot(\nu_1\rot\H+\varkappa\H\times\u-\nu_1\mathbf{j})=\0\quad & \text{in } \O,\\
\div\u=\div \H=0\quad & \text{in } \O,\\
\u=\mathbf{g},\quad\H\times\mathbf{n}=\mathbf{q}\quad & \text{on }\p\O.
\end{cases}
\end{equation}
Let $\u_0=\u_{\varepsilon_0}$ be the lifting for the boundary data ${\bf g}$ in Lemma \ref{lemma-2.1}, and $\H_0=\H_{\varepsilon_0}$ be the lifting for the boundary data ${\bf q}$ in Lemma \ref{lemma-2.2}. Here the constant $\varepsilon_0$ is to be determined.
Introducing new unknown variables $\tilde{\u}=\u-\u_0$ and $\widetilde{\H}=\H-\H_0$, we can reduce \eqref{MHD2} to the following homogeneous boundary value problem
\begin{equation}\label{MHD3}
\begin{cases}
-\nu\Delta \tilde{\u}+(\tilde{\u}\cdot\nabla)\tilde{\u}+(\tilde{\u}\cdot\nabla)\u_0+(\u_0\cdot\nabla)\tilde{\u}+\nabla p-\varkappa\rot\widetilde{\H}\times\widetilde{\H}\\
\quad-\varkappa\rot\widetilde{\H}\times\H_0-\varkappa\rot\H_0\times\widetilde{\H}={\bf F}\quad&\text{in } \O,\\
\rot(\nu_1\rot\widetilde{\H}+\varkappa\widetilde{\H}\times\tilde{\u}+\varkappa\widetilde{\H}\times\u_0+\varkappa\H_0\times\tilde{\u}-\mathbf{J})=\0\quad & \text{in } \O,\\
\div\tilde{\u}=\div \widetilde{\H}=0\quad & \text{in } \O,\\
\tilde{\u}=\0,\quad\widetilde{\H}\times\mathbf{n}=\0\quad & \text{on }\p\O,
\end{cases}
\end{equation}
where the functions ${\bf F}$ and ${\bf J}$ are given by
\begin{equation}\label{FJ}
\aligned
&{\bf F}={\bf f}+\nu\Delta \u_0-(\u_0\cdot\nabla)\u_0+\varkappa\rot\H_0\times\H_0,\\
&\mathbf{J}=\nu_1\mathbf{j}-\nu_1\rot\H_0-\varkappa\H_0\times\u_0.
\endaligned
\end{equation}
We claim that if we get a weak solution $(\tilde{\u},p,\mathbf{\widetilde{H}})$ of \eqref{MHD3}, then \eqref{MHD1}-\eqref{MHD-bc} admits a weak solution. In fact, set $\u=\u_0+\tilde{\u}$ and $\H=\H_0+\widetilde{\H}$, then $(\u,p,\H)$ solves \eqref{MHD2}. Furthermore, $(\u,p,\mathbf{E},\H)$ solves \eqref{MHD1}-\eqref{MHD-bc}, where $\mathbf{E}=\nu_1\rot\H+\varkappa\H\times\u-\nu_1\mathbf{j}$.
So we only need to prove the existence of a weak solution of \eqref{MHD3}. In the sequel, we do this by four steps.

{\it Step 1.} For any given $(\w,\D)\in H^1_0(\div0,\O)\times [H^1_{t0}(\div0,\O)\cap\mathbb{H}_D(\O)^\perp]$, we prove existence of a unique solution of the following system
\begin{equation}\label{MHD-wD}
\begin{cases}
-\nu\Delta \tilde{\u}+(\w\cdot\nabla)\tilde{\u}+(\tilde{\u}\cdot\nabla)\u_0+(\u_0\cdot\nabla)\tilde{\u}+\nabla p-\varkappa\rot\widetilde{\H}\times\D\\
\quad-\varkappa\rot\widetilde{\H}\times\H_0-\varkappa\rot\H_0\times\widetilde{\H}={\bf F}\quad&\text{in } \O,\\
\rot(\nu_1\rot\widetilde{\H}+\varkappa\D\times\tilde{\u}+\varkappa\widetilde{\H}\times\u_0+\varkappa\H_0\times\tilde{\u}-\mathbf{J})=\0\quad & \text{in } \O,\\
\div\tilde{\u}=\div \widetilde{\H}=0\quad & \text{in } \O,\\
\tilde{\u}=\0,\quad\widetilde{\H}\times\mathbf{n}=\0\quad & \text{on }\p\O,
\end{cases}
\end{equation}
where the functions ${\bf F}$ and ${\bf J}$ are defined by \eqref{FJ}.

We define a bilinear functional
\begin{align*}
&\ a((\tilde \u,\widetilde \H),(\bv,\B))=\int_{\O}\nu\nabla\tilde \u:\nabla\bv~\mathrm dx+\\
&\ \int_{\O}\left[(\w\cdot\nabla)\tilde \u+(\tilde{\u}\cdot\nabla)\u_0+(\u_0\cdot\nabla)\tilde{\u}- \varkappa\rot\widetilde\H\times\D -\varkappa \rot \widetilde\H \times \H_0 - \varkappa\rot \H_0 \times \widetilde \H\right]\cdot\bv~\mathrm dx\\
&\ +\int_{\O}( \nu_1 \rot \widetilde \H + \varkappa \D \times \tilde\u + \varkappa \widetilde\H \times \u_0 +\varkappa \H_0\times \tilde\u )\cdot\rot\B ~\mathrm dx.
\end{align*}
Then \eqref{MHD-wD} is equivalent to the formulation
$$a((\tilde\u,\widetilde\H),(\bv,\B))=\langle\F,\mathbf{v}\rangle_{H^{-1}(\O),H^1_0(\O)}+\int_{\O}\mathbf{J}\cdot\rot\B ~\mathrm dx.$$
On account that
$$\int_\Omega (\w \cdot \nabla) \tilde\u \cdot \tilde \u ~\mathrm dx=\int_\Omega (\u_0 \cdot \nabla) \tilde\u \cdot \tilde \u ~\mathrm dx= 0,$$
$$\int_\Omega (\tilde{\u}\cdot\nabla)\u_0\cdot \tilde\u~\mathrm dx=-\int_\Omega (\tilde{\u}\cdot\nabla)\tilde\u\cdot\u_0 ~\mathrm dx,$$
$$
\int_\Omega \varkappa \rot \widetilde\H \times \D \cdot \tilde\u~\mathrm dx= \int_\Omega \varkappa \D \times \tilde\u \cdot \rot\widetilde\H ~\mathrm dx,
$$
$$
\int_\Omega \varkappa \rot \widetilde\H \times \H_0 \cdot \tilde\u~\mathrm dx=\int_\Omega \varkappa\H_0\times \tilde\u \cdot \rot \widetilde\H~\mathrm dx,
$$
we have that
$$
\aligned
a((\tilde\u,\widetilde\H),(\tilde\u,\widetilde\H)) &= \int_\Omega \left(\nu |\nabla\tilde \u|^2-(\tilde{\u}\cdot\nabla)\tilde\u\cdot\u_0-\varkappa\rot\H_0\times \widetilde \H\cdot\tilde\u\right)~\mathrm dx\\
 &+ \int_\Omega\left(\nu_1 |\rot \widetilde\H|^2 +  \varkappa \widetilde\H \times \u_0 \cdot \rot\widetilde \H \right)~\mathrm dx.
\endaligned
$$

Next we \textbf{claim} that the functional $a$ is coercive. Let us estimate term by term. By H\"older's inequality and Lemma \ref{lemma-2.1}, we get
\begin{equation}\label{ineq0}
\aligned
\left|\int_\Omega (\tilde{\u}\cdot\nabla)\tilde\u\cdot\u_0 ~\mathrm dx\right|&\leq \|\tilde{\u}\|_{L^6(\Omega)} \|\nabla\tilde\u\|_{L^2(\Omega)} \|\u_0\|_{L^3(\Omega)}\leq C_1\varepsilon_0  \|\mathbf g\|_{H^{1/2}(\partial\Omega)}\|\nabla\tilde\u\|_{L^2(\Omega)}^2.
\endaligned
\end{equation}
By H\"older's inequality, Lemma \ref{lemma-2.1} and  Lemma \ref{lemma-2.3}, it holds that
\begin{equation}\label{ineq2}
\aligned
\left|\int_\Omega \varkappa \widetilde\H \times \u_0 \cdot \rot\widetilde \H~\mathrm dx\right|&\leq \varkappa \|\widetilde \H\|_{L^6(\Omega)} \|\u_0\|_{L^3(\Omega)} \|\rot \widetilde\H\|_{L^2(\Omega)}\\
  &\leq C_2 \varkappa \varepsilon_0 \|\mathbf g\|_{H^{1/2}(\partial\Omega)}\|\rot\widetilde\H\|_{L^2(\Omega)}^2,
\endaligned
\end{equation}
where Lemma \ref{lemma-2.3} is used in the last inequality.
Note that
$$\rot (\A \times \B) = (\div \B) \A - (\div \A) \B + (\B\cdot \nabla)\A - (\A\cdot\nabla)\B.$$
Integrating by parts and using H\"older's inequality, we have that
\begin{equation}\label{ineq3}
\aligned
 \left|\int_\Omega \varkappa \rot \H_0 \times \widetilde \H \cdot \tilde\u~\mathrm dx\right| &= \left|\varkappa \int_\Omega \rot\H_0 \cdot (\widetilde \H \times\tilde\u)~\mathrm dx\right|\\
 &= \left|\varkappa \int_\Omega \H_0 \cdot \rot(\widetilde\H\times\tilde\u)~\mathrm dx+\varkappa \int_{\p\Omega}  (\widetilde\H\times\tilde\u)\times \mathbf{n}\cdot\H_0 ~\mathrm dx\right|\\
&= \left|\varkappa \int_\Omega \H_0 \cdot \rot(\widetilde\H\times\tilde\u)~\mathrm dx\right|\\
&= \left|\varkappa \int_\Omega \H_0 \cdot \left[(\div \tilde\u)\widetilde\H-(\div \widetilde\H)\tilde\u+(\tilde\u\cdot\nabla)\widetilde\H - (\widetilde\H\cdot\nabla)\tilde\u \right]~\mathrm dx\right|\\
&= \left|\varkappa \int_\Omega \H_0 \cdot \left[(\tilde\u\cdot\nabla)\widetilde\H - (\widetilde\H\cdot\nabla)\tilde\u \right]~\mathrm dx\right| \\
 &\leq \varkappa \|\H_0\|_{L^3(\Omega)} \left( \|\tilde\u\|_{L^6(\Omega)} \|\nabla\widetilde \H\|_{L^2(\Omega)} + \|\widetilde\H\|_{L^6(\Omega)}\|\nabla \tilde\u\|_{L^2(\Omega)} \right)\\
 &\leq C_3\varkappa \varepsilon_0 \|\mathbf q\|_{H^{1/2}(\partial\Omega)} \left( \|\nabla \tilde\u\|_{L^2(\Omega)}^2 + \|{\rot\widetilde\H}\|_{L^2(\Omega)}^2 \right),
 \endaligned
\end{equation}
where we have used $\tilde\u=\0$ on $\p\O$ in the third equality, $\div\tilde\u=\div\widetilde\H=0$ in the fifth equality, Lemma \ref{lemma-2.2} and Lemma \ref{lemma-2.3} in the last inequality. Combining inequalities \eqref{ineq0}, \eqref{ineq2} and \eqref{ineq3}, we conclude that
\begin{multline*}
  a((\tilde\u,\widetilde\H),(\tilde\u,\widetilde\H)) \geq \nu \|\nabla \tilde\u\|_{L^2(\Omega)}^2 + \nu_1 \|\rot \widetilde \H\|_{L^2(\Omega)}^2 \\
   - (C_1+C_2\varkappa+C_3\varkappa) \varepsilon_0 \left( \|\mathbf g\|_{H^{1/2}(\partial\Omega)} + \|\mathbf q\|_{H^{1/2}(\partial\Omega)} \right)\left( \|\nabla\tilde\u\|_{L^2(\Omega)}^2 + \|\rot\widetilde\H\|_{L^2(\Omega)}^2 \right).
\end{multline*}
So the bilinear functional $a$ is obviously coercive if $\varepsilon_0$ is selected small enough. The claim is then proved.

By Lax-Milgram theorem and De Rham's theorem, we obtain the existence of unique weak solution $(\tilde\u, p, \widetilde\H)$ to \eqref{MHD-wD} in $H_0^1(\div 0,\Omega) \times L^2(\Omega)/\mathbb{R} \times [H^1_{t0}(\div0,\O)\cap \mathbb{H}_D(\O)^\perp]$, with
\begin{equation*}
\|\nabla\tilde\u\|_{L^2(\Omega)} + \|\rot\widetilde\H\|_{L^2(\Omega)} \leq C(\Omega,\nu,\nu_1,\varkappa,\varepsilon_0,\mathbf{g},\mathbf{q}) \left( \|F\|_{H^{-1}(\Omega)} + \|J\|_{L^2(\Omega)} \right).
\end{equation*}
In addition, it is easy to check that
\begin{equation*}
\|F\|_{H^{-1}(\Omega)} \leq C(\Omega,\varepsilon_0) \left( \|\mathbf f\|_{H^{-1}(\Omega)} + \nu \|\mathbf g\|_{H^{1/2}(\Omega)} + \|\mathbf g\|_{H^{1/2}(\Omega)}^2 + \varkappa \|\mathbf q\|_{H^{1/2}(\Omega)}^2  \right),
\end{equation*}
and
\begin{equation*}
\|J\|_{L^2(\Omega)} \leq C(\Omega,\varepsilon_0) \left( \nu_1 \|\mathbf j\|_{L^2(\Omega)} + \nu_1 \|\mathbf q\|_{H^{1/2}(\Omega)} + \varkappa \|\mathbf g\|_{H^{1/2}(\Omega)}\|\mathbf q\|_{H^{1/2}(\Omega)}  \right).
\end{equation*}
Therefore, we get that
\begin{equation}\label{C-K}
\|\tilde\u\|_{H^1(\Omega)} + \|\widetilde \H\|_{H^1(\Omega)} \leq  C\left( \|\mathbf f\|_{H^{-1}(\Omega)} + \|\mathbf j\|_{L^2(\Omega)} + 1 + \|\mathbf g\|_{H^{1/2}(\Omega)}^2 + \|\mathbf q\|_{H^{1/2}(\Omega)}^2   \right),
\end{equation}
where $C=C(\Omega,\nu,\nu_1,\varkappa,\varepsilon_0,\mathbf{g},\mathbf{q})$.

{\it Step 2.} For any given $(\w,\D)$ above, define an operator $\mathrm{T}$ by $\mathrm{T}(\w,\D)=(\tilde\u,\widetilde\H)$.
Let $K$ be the right hand side of \eqref{C-K}. We define
$$
\aligned
\mathcal{D}=\{(\w,\D)\in H^1_0(\div0,\O)\times [H^1_{t0}(\div0,\O)\cap \mathbb{H}_D(\O)^\perp]:\|\w\|_{H^1(\O)}+\|\D\|_{H^1(\O)}\leq K\}.
\endaligned
$$
Obviously, $\mathcal{D}$ is a bounded, closed and convex subset of $H^1_0(\div0,\O)\times [H^1_{t0}(\div0,\O)\cap \mathbb{H}_D(\O)^\perp]$.
Moreover, $\mathrm{T}$ maps $\mathcal{D}$ into itself.

{\it Step 3.} We show that $\mathrm{T}$ is continuous and compact from $\mathcal{D}$ into $\mathcal{D}$. First, we prove that $\mathrm{T}$ is continuous. Assume that $(\w_k,\D_k)$, $(\w,\D)\in\mathcal{D}$ and
\begin{equation}\label{assum}
\text{$\w_k\rightarrow\w$, $\D_k\rightarrow\D$ in $H^1(\O,\mathbb{R}^3)$ as $k\rightarrow\infty$.}
\end{equation}
Let $(\tilde\u, p, \widetilde\H)$ be the weak solution of \eqref{MHD-wD} and $(\tilde\u_k, p_k, \widetilde\H_k)$ be the weak solution of \eqref{MHD-wD} with $(\w,\D)$ replaced by $(\w_k,\D_k)$. If we set
\begin{equation*}
  \mathbf v_k = \tilde\u_k -\tilde\u, \qquad \pi_k = p_k-p, \qquad \B_k=\widetilde\H_k - \widetilde\H,
\end{equation*}
it is easy to check that $( \mathbf v_k, \pi_k,\B_k)$ satisfies the following system
\begin{equation*}
\begin{cases}
-\nu \Delta \mathbf v_k + (\w_k\cdot\nabla)\tilde\u_k - (\w\cdot\nabla)\tilde\u + (\mathbf v_k\cdot\nabla)\u_0 + (\u_0\cdot\nabla)\mathbf v_k + \nabla \pi_k\\
 -\left(\varkappa\rot\widetilde\H_k \times \D_k - \varkappa \rot\widetilde\H \times \D\right) - \varkappa\rot \B_k \times \H_0 -\varkappa\rot \H_0\times \B_k=\0  &\text{ in }\Omega, \\
\rot\left(\nu_1\rot\B_k +\varkappa \D_k\times\tilde\u_k - \varkappa \D \times \tilde\u + \varkappa\B_k\times\u_0 + \varkappa \H_0\times\mathbf v_k\right) = \0  &\text{ in }\Omega, \\
\div\mathbf v_k = \div\B_k = 0 &\text{ in }\Omega,\\
\mathbf v_k =\0, \ \ \B_k\times \n = \0&\text{ on }\partial\Omega.
\end{cases}
\end{equation*}
Taking $(\mathbf v_k,\B_k)$ as a test function for the above system, and noting that
\begin{gather*}
  (\w_k\cdot\nabla)\tilde\u_k - (\w\cdot\nabla)\tilde\u = (\w_k\cdot\nabla)\mathbf v_k + [(\w_k-\w)\cdot\nabla]\tilde\u,\\
  \rot\widetilde\H_k \times\D_k -  \rot\widetilde\H \times \D = \rot \B_k \times \D_k + \rot \widetilde\H \times (\D_k-\D),\\
  \D_k\times\tilde\u_k -  \D \times \tilde\u = \D_k \times\mathbf v_k + (\D_k-\D)\times\tilde\u,
\end{gather*}
we obtain that
\begin{equation*}
\aligned
  &\int_\Omega \left[\nu |\nabla\mathbf v_k|^2 + \nu_1 |\rot\B_k|^2 + (\mathbf v_k\cdot\nabla)\u_0 \cdot\mathbf v_k - \varkappa \rot \B_k \times \H_0 \cdot\mathbf v_k \right]~\mathrm dx\\
  &+ \int_\Omega [ - \varkappa \rot \H_0 \times \B_k \cdot \mathbf v_k+\varkappa \B_k \times \u_0 \cdot \rot B_k  + \varkappa \H_0\times \mathbf v_k \cdot \rot\B_k]~\mathrm dx\\
  &= \int_\Omega \{\varkappa \rot\widetilde \H \times (\D_k-\D)\cdot\mathbf v_k - [(\w_k-\w)\cdot\nabla]\tilde\u\cdot\mathbf v_k - \varkappa (\D_k-\D )\times \tilde\u \cdot\rot\B_k\}~\mathrm dx.
\endaligned
\end{equation*}
By a similar procedure as in Step 1 and using H\"older's inequality, we can get that
\begin{equation*}
\aligned
  &C(\Omega,\nu,\nu_1,\varkappa,\varepsilon_0,\mathbf{g},\mathbf{q})  \left( \|\nabla\mathbf v_k\|_{L^2(\Omega)}^2 + \|\rot\B_k\|_{L^2(\Omega)}^2 \right)\\
  &\leq \varkappa\|\rot\widetilde \H\|_{L^2(\Omega)} \|\D_k-\D\|_{L^3(\Omega)} \|\mathbf v_k\|_{L^6(\Omega)} + \|\w_k -\w \|_{L^3(\Omega)} \|\nabla \tilde \u\|_{L^2(\Omega)}\|\mathbf v_k\|_{L^6(\Omega)} \\
 &\quad+\varkappa\|\D_k-\D\|_{L^3(\Omega)} \|\tilde\u\|_{L^6(\Omega)}\|\rot\B_k\|_{L^2(\Omega)}.
  \endaligned
\end{equation*}
Hence it follows that
$$
\aligned
&\|\nabla\mathbf v_k\|_{L^2(\Omega)} + \|\rot\B_k\|_{L^2(\Omega)}\\
&\leq C(\|\rot\widetilde \H\|_{L^2(\Omega)} \|\D_k-\D\|_{L^3(\Omega)} + \|\w_k -\w \|_{L^3(\Omega)} \|\nabla \tilde \u\|_{L^2(\Omega)}+\|\D_k-\D\|_{L^3(\Omega)} \|\tilde\u\|_{L^6(\Omega)}),
\endaligned
$$
where the constant $C$ depends on $\Omega,\nu,\nu_1,\varkappa,\varepsilon_0,\mathbf{g},\mathbf{q}$.
Together with \eqref{C-K} and \eqref{assum}, we conclude by Lemma \ref{lemma-2.3} that
\begin{equation*}
\tilde\u_k \to \tilde\u\quad \text{and} \quad \widetilde\H_k \to \widetilde\H \quad \text{in }H^1(\Omega,\mathbb R^3)\text{ as } k\rightarrow\infty,
\end{equation*}
which implies the continuity of the operator $\rm T$.

And then we show that $\mathrm{T}$ is compact from $\mathcal{D}$ into $\mathcal{D}$. Assume that $(\w_k,\D_k)\in\mathcal{D}$. Then there exist $(\w,\D)\in\mathcal{D}$ and a subsequence of $\{(\w_k,\D_k)\}$, still denoted by $\{(\w_k,\D_k)\}$ to simplify the notation, satisfying
$$\w_k\rightharpoonup\w,\; \D_k\rightharpoonup\D\text{ in } H^1(\O,\mathbb{R}^3) \text{ and } \w_k\rightarrow\w,\;\D_k\rightarrow\D \text{ in } L^3(\O,\mathbb{R}^3) \text{ as } k\rightarrow\infty.$$
Similarly to the proof of continuity of $\mathrm{T}$, we obtain
\begin{equation*}
\tilde\u_k \to \tilde\u\quad \text{and} \quad \widetilde\H_k \to \widetilde\H \quad \text{in }H^1(\Omega,\mathbb R^3)\text{ as } k\rightarrow\infty.
\end{equation*}

{\it Step 4.} Finally, we use Schauder's fixed point theorem and conclude that $\mathrm{T}$ has a fixed point $(\tilde\u,\widetilde\H)\in \mathcal{D}$. Then by De Rham's
theorem there exists a function $p\in L^2(\O)/\R$ such that $(\tilde\u,p,\widetilde\H)$ be a weak solution of \eqref{MHD-wD} with $(\w,\D)$ replaced by $(\tilde\u,\widetilde\H)$. So we get a weak solution of \eqref{MHD3}.
\end{proof}

\bigskip

\subsection*{Acknowledgements.}
The first author would like to thank Dr. Yong Zeng for bringing the reference \cite{Alekseev04} in 2016. This work does not have any conflicts of interest. Zhang was supported by the National Natural Science Foundation of China Grant No. 11901003 and Anhui Provincial Natural Science Foundation Grant No. 1908085QA28.
Zhao was supported by the National Natural Science Foundation of China Grant No. 11971169.

\end{document}